\documentclass[10pt]{article}
\usepackage{graphicx}
\usepackage{amssymb,amsmath,latexsym,amscd}
\usepackage{amsthm}

\newtheorem{theorem}[equation]{Theorem}
\newtheorem{corollary}[equation]{Corollary}
\newtheorem{lemma}[equation]{Lemma}
\newtheorem{explanation}[equation]{Explanation}

\theoremstyle{definition}\newtheorem{definition}[equation]{Definition}
\theoremstyle{definition}\newtheorem{remark}[equation]{Remark}
\theoremstyle{remark}
\numberwithin{equation}{section}

\newcommand{\Q}{{\mathbb Q}}

\newcommand{\Z}{\mathbb{Z}}
\newcommand{\KK}{\mathbb{K}}
\newcommand{\LL}{\mathbb{L}}
\newcommand{\MM}{\mathbb{M}}
\newcommand{\FF}{\mathbb{F}}

\newcommand{\QQ}{\mathbb{Q}}

\newcommand{\ch}{\rm Char}

\newcommand{\Gal}{{\mathcal Gal}}

\begin{document}
\title{}
\author{Alexander Stolin$^{1}$}
\date{}
\title{Vandiver's Conjecture via K-theory}

\maketitle

$^1$ Department of Mathematical Sciences, Chalmers University of Technology and the University of Gothenburg, 412 96 Gothenburg, Sweden.

\begin{abstract}
From Wikipedia: "In mathematics, the Kummer--Vandiver conjecture, or Vandiver conjecture, states that a prime p does not divide the class number h(R) of the maximal real subfield R  of the p-th cyclotomic field. The conjecture was first made by Ernst Kummer in 1849 December 28 and 1853 April 24 in letters to Leopold Kronecker, reprinted in (Kummer 1975, pages 84, 93, 123--124), and independently proposed around 1920 by Philipp Furtwängler and Harry Vandiver. As of 2011, there is no particularly strong evidence either for or against the conjecture and it is unclear whether it is true or false, though it is likely that counterexamples are very rare."

Kummer verified the conjecture for p less than 200, and Vandiver extended this to p less than 600. Harvey (2008) extended this to primes less than $163\times 10^6$.

In this paper I would like to to prove Vandiver's conjecture and indicate some consequences including
the first case of Fermat's Great Theorem, some properties of the Iwasawa numbers
and to present an exact formula for the p-Sylow subgroup of the class group of
$\Q (\zeta_n ),$   where $\zeta_n^{p^{n+1}}=1$ as an abelian group ({\bf OBS!}
assuming that Vandiver's conjecture is true, a formula for the class group above
as a $\Gamma$-module is well-known. However, the structure of that group as an abelian group
remained unknown).

 \noindent \\
\end{abstract}

\section{Introduction, necessary facts from K-theory}
In what follows we will need a number of facts about Picard groups proved in \cite{S1}.

We will use the following notations.
$C_2$ will be the cyclic group of order $p^2$, where $p$ is an odd prime number.
$\zeta=\zeta_1$ will be a $p$-th root of unity, while $\zeta_2$ will be a primitive $p^2$-root of unity.
\subsection{Review of necessary results of \cite{S1}}
Let us consider the Picard group of the integer group ring $\mathbb{Z} C_2$.
The first observation is
that $Pic(\mathbb{Z} C_2)\cong Pic(A)$, where $A$ can be presented as a Cartesian product of
$\mathbb{Z} [\zeta_1 ]$ and $\mathbb{Z} [\zeta_2]$ over the local ring $\Z[\zeta_1]/(p):=F$.

The corresponding  Mayer--Vietoris sequence reads as
$$
0\to V\to Pic (A)\to Pic (\mathbb{Z} [\zeta_2])\oplus Pic (\mathbb{Z} [\zeta_1])\to 0.
$$
Here $Pic$, the Picard group, is the same as the projective class group or simply the {\it class group} for Dedekind rings.
We will use the standard notation $Cl (D).$

The group $V$ was computed in \cite{KM} for the primes satisfying Vandiver's conjecture.
However, similar computations can be done for any prime.
Let us do them. Let $E_i,\ i=1,2$ be the group of units of $\Z[\zeta_i].$ Abusing notations let us denote their images in
$U(\Z[\zeta_1]/(p))$ also by $E_i,\ i=1,2$. Then by definition $V= U((\Z[\zeta_1])/(p))/(E_1 \times E_2)$. Here $U(R)$ is the group of units
of an abelian ring $R.$

The following result was proved in \cite{S1}:
\begin{theorem}
$V= U(\Z[\zeta_1]/(p))/E_1$.
\end{theorem}
The proof is based on the following useful result:
\begin{theorem}
Let $j_i, \ i=1,2$ be canonical maps $j_i:\ \Z[\zeta_i]\to\Z[\zeta_1]/(p)$ and let $N:\ \Z[\zeta_2]\to\Z[\zeta_1]$ be the norm map.
Then $j_2 (a)=j_1 (N(a)).$

\end{theorem}
The structure of $U(\Z[\zeta_1]/(p))$ is well-known:
$$U(\Z[\zeta_1]/(p))\cong \FF_p^*\oplus \FF_p^{p-2} .$$
Here, $\FF_p $ is the prime field of characteristic $p$ and $\FF_p^*$ is the group of its invertible elements.
Now, let us introduce an important number $r_0$. Let $S_i,\ i=1,2$ be the Sylow $p$-part of the class group of $\Z[\zeta_i].$
Let us denote the subgroup of the group ${\ch} (S_1/S_1^p)$ generated by all  $\epsilon\in U(\Z[\zeta_1]):\ \epsilon\equiv 1\ mod\ (p)$
by $({\ch} (S_1/S_1^p))_\epsilon .$
\begin{definition}
$r_0 = log_p (\#({\ch} (S_1/S_1^p)_\epsilon))=\newline
log_p (\#\{\{\epsilon\equiv 1\ mod\ (p)\}/\{\epsilon\equiv 1\ mod\ (\zeta_1 -1)^2\}^p \}) .$
\end{definition}
Now we can formulate an important result:
\begin{theorem}
$V\cong \FF_p^{\frac{p-3}{2} +r_0}.$
\end{theorem}

Another important observation was made in \cite{S1}: there exists a splitting map
$Cl (\mathbb{Z} [\zeta_2])\to Pic (A) .$ Therefore,
$Pic ( A)= Pic (\mathbb{Z} [\zeta_2])\oplus B$ for a certain group $B$, which will be described explicitly in the next subsection. The Mayer--Vietoris exact sequence reads now as
$$
0\to V\to B\to  Pic (\mathbb{Z} [\zeta_1])\to 0.
$$
We continue to review necessary results of \cite{S1}.
\subsection{Invertible modules over $A$ }
By Milnor's theory of projective modules  over Cartesian products (see \cite{M}), any invertible module over the ring $A$
can be presented as $M(\alpha_1,\alpha_2, h) ,$ where $\alpha_k\in\Z[\zeta_k],\ k=1,2$ and $h\in U(F).$ The  facts below
were proved in \cite{S1}:
\begin{enumerate}
\item $M(\alpha_1,\alpha_2, h)\cong M(q_1 \alpha_1,q_2 \alpha_2, j_1 (q_1^{-1})hj_2 (q_2) ,$ where $q_k \in \Z[\zeta_k]$
are such that $(q_k ,p)=1 $;
\item $M(\alpha_1,\alpha_2, h)$ is free if and only if $\alpha_k=(r_k),\ r_k\in\Z[\zeta_k]$ and
$j_1 (r_1^{-1})hj_2 (r_2)=j_1 (\epsilon),\ \epsilon\in U(\Z[\zeta_1])$;
\item $M(\alpha_1,\alpha_2, h)\otimes M(\beta_1,\beta_2, g)\cong M(\alpha_1\beta_1,\alpha_2\beta_2, hg) ;$
\item the map $\psi: Cl(\Z[\zeta_2])\to Pic(A)$ defined as $\phi(\alpha)=M(N(\alpha),\alpha, 1)$ ($N$ is the norm map
extended to ideals of $\Z[\zeta_2]$) splits the canonical projection $Pic(A)\to Cl(\Z[\zeta_2]);$
\item $Pic(A)\cong Cl(\Z[\zeta_2])\oplus B$ for some group $B$ and $V\subset B.$
\item $B$ is generated by invertible modules of the form $M(\alpha,\Z[\zeta_2], h)$ .
\end{enumerate}
Note that the Sylow p-component of $Pic (\mathbb{Z} [\zeta_1])=Cl(\Z[\zeta_1])$ is exactly $S_1 .$
We will denote the Sylow $p$-part of $B$ by $B_p .$ Then, the essential part of the exact Mayer--Vietoris reads as follows:
$$0\to V\to B_p \to S_1 \to 0.$$
\subsection{Numbers $R$, $r$, and fine structure of $B_p$}
Let us denote the number of $\Z_p$-generators of the group $S_1$ by $R.$
Let us denote by $L\subset S_1$ the subgroup generated by the ideal classes $\alpha\in S_1$ such that
$M(\alpha,\Z[\zeta_2], h)$ has exponent $p$ in $Pic(A).$ Of course, $L$ is a subgroup of $S_1$
and $\alpha$ has exponent $1$ (this means that $\alpha^p=1$) in $S_1 .$
\begin{lemma}
$L$ consists of elements $\alpha\in S_1$ such that $\alpha^p=(q)$ and $q\equiv \epsilon\in U(\Z[\zeta_1])\ mod (p).$
\end{lemma}
\begin{remark}
Changing $q$, we can set $\epsilon=1.$
\end{remark}
\begin{proof}
A proof easily follows from the properties of $M(\alpha_1,\alpha_2,h)$ above.
\end{proof}
The following result was proved in \cite{S1}:
\begin{lemma}\label{L}
$L\cong \ch((S_1/S_1^p)/\ch(S_1/S_1^p)_\epsilon).$
\end{lemma}
Let us denote the number of of $\Z_p$-generators of $L$ by $r.$
\begin{corollary}
$R-r=r_0.$

\end{corollary}
\begin{corollary}\label{cor}
$B_p\cong\FF^{(p-3)/2} \oplus B_1 ,$ where $B_1$ can be described in the following way:
it is a $p$-group, its number $\Z_p$-generators is $R$. They are in a one-to-one correspondence with
generators $\alpha_i,\ 1\leq i\leq R$ of $S_1$. We denote them by $\beta_1,\cdots,\beta_R$.
For $1\leq i\leq r$ we have $exp (\beta_i)=exp (\alpha_i)$ and for $r+1\leq i\leq R$ $exp (\beta_i)=1+exp (\alpha_i)$,
where $exp ( x)=n$ means that $x^{p^n}=1$ but $x^{p^{n-1}}\neq 1.$
\end{corollary}

The choice of $\Z_p$-generators of the group $B_1$ is not canonical. However,  generators of $\FF^{(p-3)/2} $
can be chosen in a canonical way.

\begin{theorem}
Let $S_p$ be a subgroup of $S_1$ consisting of elements of exponent $1.$
Then the map $S_p\to V$ determined by $\alpha\to q$, where $\alpha^p =(q)$ is well-defined and its kernel is $L.$
\end{theorem}
\begin{proof}
If we substitute $\alpha$ by $r\alpha,\ r\in \Q(\zeta_1),$ then $q\to r^p q$. Another possible choice of $q$ is
$q\epsilon$. These changes will not affect the image of $\alpha$ in $V.$ The statement about the kernel follows
from the properties of $M(\alpha_1,\alpha_2,h)$ above.
\end{proof}
\begin{corollary}
The image of $S_p$\ in $ V$ is isomorphic to $\FF_p^{r_0}$.
\end{corollary}
We remind the reader that $V\cong\FF_p^{\frac{p-3}{2} +r_0}$ and the group $Gal(\Q(\zeta_1)/\Q):=G$ acts on $V.$
The image of $S_p$ is $G$-invariant and therefore, the summand $\FF_p^{\frac{p-3}{2}}\subset V\subset B_p$ can be
chosen canonically as a $G$-invariant complement to the image of $S_p$ in $V.$
\begin{remark}
The complex conjugation $\sigma\in G$ acts on $V.$ Consequently, $V=V_+ \oplus V_- $ with obvious $\pm$-parts $V_\pm .$
It easily follows from Kummer's Lemma ($\epsilon=\zeta_1^k \epsilon_{real}$) that $V_-\cong\FF^{\frac{p-3}{2}}$ while
$V_+\cong \FF^{r_0}$. However, it is not difficult to see that $V_+\neq Im (S_p) .$
\end{remark}
\subsection{Action of complex conjugation and further relations between $R$- and $r$-numbers}
Let $r_\pm$ be the number of $\Z_p$-generators of $L_\pm$ and let $R_\pm$ be the number of $\Z_p$-generators
of $S_1 .$ Here, $L=L_+\oplus L_-$ and $S_1=(S_1)_+\oplus (S_1 )_-$ are defined by action of the complex conjugation $\sigma .$
\begin{lemma}
$\ch (S_1/S_1^p)_\epsilon \subset (\ch (S_1/S_1^p))_+ =\ch((S_1/S_1^p)_-)$.
\end{lemma}
\begin{proof}
The units defining $\ch (S_1/S_1^p)_\epsilon$ are real. It is a consequence of Kummer's Lemma mentioned above.
\end{proof}
\begin{corollary}
$r_- =R_+,\ R_-=r_+ +r_0 .$
\end{corollary}
\begin{proof}
Using \ref{L} and the lemma above, we see that $L_- =(\ch(S_1/S_1^p))_- =\ch((S_1/S_1^p)_+)$ what implies the first equality.
The second can be proved analogously, alternatively it follows from  equalities $R=r+r_0,\ R=R_+ +R_-, \ r=r_+ +r_- .$
\end{proof}
\begin{corollary}\bf{First inequality.}

$R_- -r_- \leq r_0$
\end{corollary}
\begin{proof}
$r_0 =R-r=(R_+ -r_+)+(R_- -r_-) \geq R_- -r_-.$
\end{proof}
Now, we can reformulate \ref{cor}:
\begin{corollary}
\begin{itemize}
\item We have two exact sequences
$$0\to V_-=\FF_p^{(p-3)/2}\to (B_p)_- \to (S_1 )_-\to 0\ {\rm and}$$
$$0\to V_+=\FF_p^{r_0}\to (B_p)_+ \to (S_1)_+ \to 0.$$
\item $(B_p)_- =\FF_p^{(\frac{p-3}{2}+r_{-}) -R_-}\oplus (B_1)_-$
\item The group $(B_1)_-$ can be described as follows: it has $R_-$
$\Z_p$-generators, which are in a one-to-one correspondence with generators $\alpha_1,\cdots, \alpha_{R_-}$ of $(S_1)_-$.
Furthermore, $r_-$ generators of $(B_1)_-$ have the same exponent as the corresponding generator
of $(S_1)_-$ and $R_- -r_-$ generators have exponent $exp(\alpha_i) +1 .$
\item $(B_p)_+$ has a similar description but we will not need it.
\end{itemize}
\end{corollary}
\section{$B_p$ as a Galois group and the second inequality}
Following ideas of \cite{KM}, let us define the ray ideal group $H$ as the group
of those principal ideals of $\Q(\zeta_1)$, which possess a generator $g$ such that
$g\equiv 1 \ mod(p) .$

Let us denote $\Q(\zeta_n), \zeta_n^{p^n}=1$ by $\KK_{n-1} .$ Sometimes we denote $\KK_0$ simply by $\KK.$
Let $\MM/\KK$ be the Sylow $p$-part of the ray class field extension associated with the group $H.$
Hence, $\MM/\KK$ is an abelian extension with the Galois group $Gal (\MM/\KK)\cong (I_0 (\KK)/H)_p$,
where $I_0 (\KK)$ is the group of ideals of $\KK$ prime to $p$ and
$(I_0 (\KK)/H)_p$ is the Sylow $p$-subgroup
of the finite group $I_0 (\KK)/H$ (the group of all ideals of
$\KK$ will be denoted by $I(\KK)$).

We remind the reader that groups $B$ and $B_p$ were defined in subsection 1.2.
\begin{theorem}
The map  $M(\alpha,\Z[\zeta_2], h)\to h^{-1}\alpha$ defines isomorphisms   $B\cong I_0 (\KK)/H$ and $B_p \cong (I_0 (\KK)/H)_p .$
\end{theorem}
\begin{proof}
Clearly this map is surjective. Let us prove that the map is injective. Indeed, its kernel consists of elements such
that $h^{-1}\alpha$ is a principal ideal, which possesses a generator $g\equiv 1 \ mod(p) .$
However, it follows from the properties of invertible modules over the ring $A$ (subsection 1.2) that such
$M(\alpha,\Z[\zeta_2], h)$ is free.
\end{proof}
\begin{corollary}
Let $\LL$ be the Sylow $p$-part of the Hilbert class field of $\KK .$
\begin{itemize}
\item Then $Gal(\LL/\KK)\cong S_1$,
$Gal(\MM/\LL)\cong V ,$ and the Mayer-Vietoris exact sequence $0\to V\to B_p\to S_1$ becomes the exact
sequence of the Galois groups of the tower of field extensions $\KK\subset\LL\subset\MM.$
\item Consequently, there are exactly $R-r$ cyclic extensions $\MM_{\alpha}/\KK$ such that $\MM_\alpha\subset\MM$,
$Gal(\MM_{\alpha}/\KK)$
is a cyclic group of order $p^k ,k\geq 2$ and  $Gal(\MM_{\alpha}/\MM_\alpha \cap \LL)$ is cyclic of order
$p$ or if you wish $Gal((\MM_\alpha \cap \LL)/\KK)$ is cyclic of
order $p^{k-1}$.
\item Let $ \MM_-$
be the extension of $\KK$ with the Galois group $(B_p)_- .$
Then there are $R_- -r_-$ cyclic extensions $\MM_{\alpha}/\KK$ such that $\MM_\alpha\subset(\MM)_-$, $Gal(\MM_{\alpha}/\KK)$
is a cyclic group of order $p^k ,k\geq 2$ and  $Gal(\MM_{\alpha}/\MM_\alpha \cap \LL)$ is cyclic of order
$p.$
\end{itemize}

\end{corollary}

Now, we will need the diagram of field extensions as illustrated below.
\begin{center}
  \includegraphics[scale=0.5]{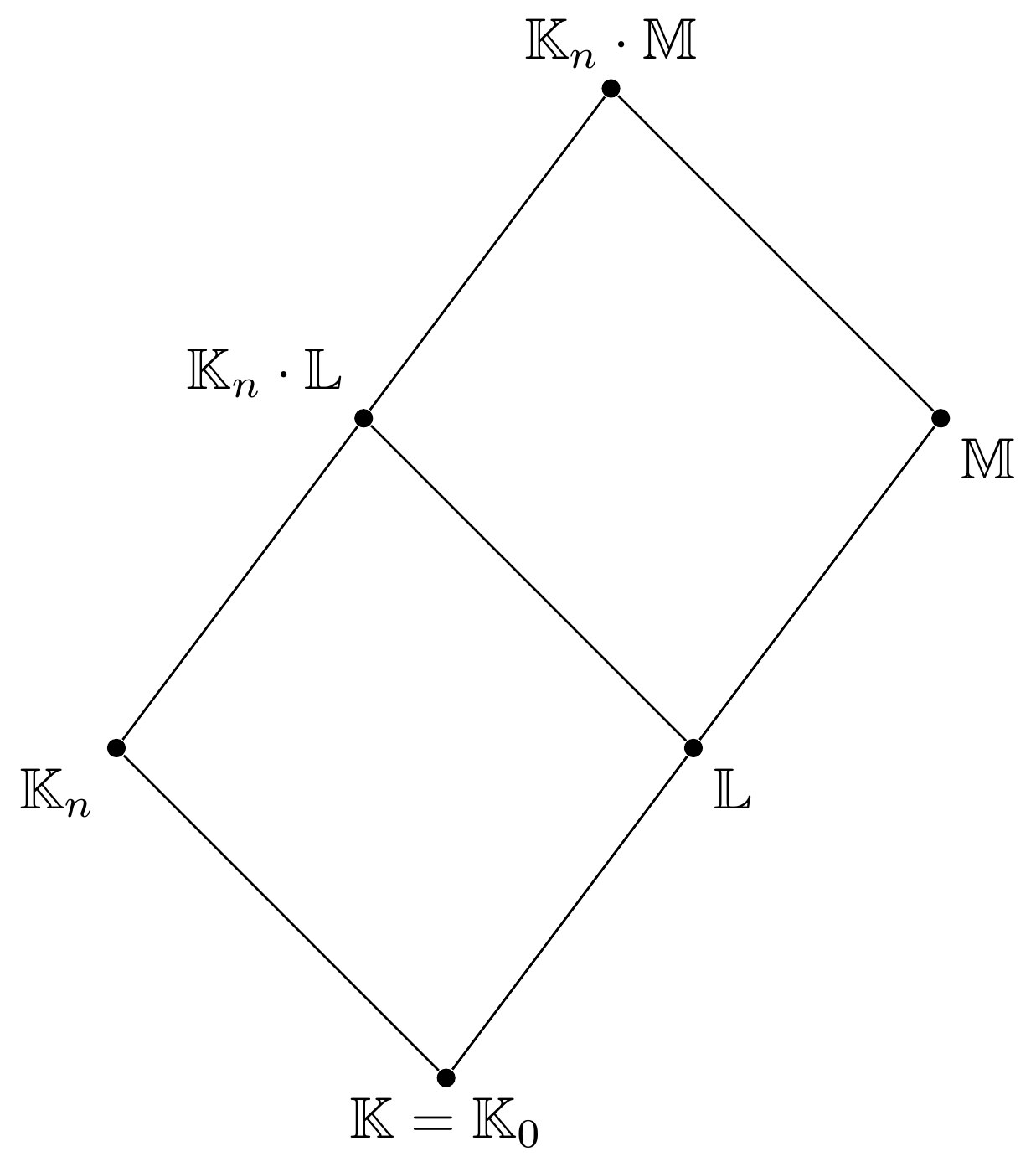}
\end{center}

\begin{lemma}
$\KK_n \cap\MM=\KK$
\end{lemma}
\begin{proof}
We follow the proof of Lemma 4.4 of \cite{KM}.
If $\KK_n \cap\MM\neq\KK$, then $\KK_n \cap\MM\supset\KK_1$.
Therefore, it is sufficient to prove that $\KK_1$
{\it is not contained in} $\MM.$ Indeed, $\MM$ is associated with the ray group $H$ defined above. As it is well-known, $\KK_1$ is a class field over
$\KK=\KK_0$ associated with the ray group
$$P=\{  \alpha\in I(\KK) :\ Norm(\alpha)\equiv 1\ mod (p^2)\} ,$$
(see \cite{KM} and references therein).
Here, the norm $Norm$ acts from $\KK$ to $\QQ$.Hence, we must prove that $H\not\subset P .$

Let us choose $\beta=(1+p) .$ Clearly, $\beta\in H.$ On the other hand,
$$Norm (1+p)=(1+p)^{p-1}\equiv \ 1+(p-1)p \ mod(p^2)\not\equiv 1\ mod(p^2) . $$
Hence, $\beta\not\in P$ and $H\not\subset P$. The lemma is proved.
\end{proof}
\begin{corollary}
$Gal(\KK_n\cdot\MM/\KK_n)\cong Gal(\MM/\KK) .$
\end{corollary}
The main aim of this section is to prove {\bf the second inequality}
namely $r_0\leq R_- -r_-$. To do this, we will study
$(Gal(\KK_n\cdot\MM/\KK_n))_- =(Gal (\MM/\KK))_- =(B_p)_- .$

We make an {\bf important}
\begin{remark}
Since $p$ is odd, we have $(B_p)_-\cong ((I_0 (\KK)/H)_p )_-\cong
\newline ((I_0 (\KK)/\tilde{H})_p )_-$, where $\tilde{H}$
is the group
of those principal ideals of $\KK_0=\Q(\zeta_1)$, which possess a generator $g$ such that
$g\equiv 1 \ mod(\zeta_1 -1)^p .$
Indeed, there is an obvious surjection $(I_0 (\KK)/\tilde{H})_p\to (I_0 (\KK)/H)_p$
and its kernel is the ideal $(p+1)$, which is contained in $((I_0 (\KK)/\tilde{H})_p)_+$.
Hence, the odd parts are isomorphic.
\end{remark}
From now on we will use notation $H_n$ for the group of those principal
ideals of $\KK_n=\Q(\zeta_{n+1})$, which possess a generator $g$ such that
$g\equiv 1 \ mod(\zeta_{n+1} -1)^{p^{n+1}} $ $(H_0=\tilde{H}).$

Let $\MM_n/\KK_n$ be the Sylow $p$-part of the ray class field extension associated with the group $H_n.$
Hence, $\MM_n/\KK_n$ is an abelian extension with the Galois group $Gal (\MM_n/\KK_n)\cong (I_0 (\KK_n)/H_n)_p$,
where $I_0 (\KK_n)$ is the group of ideals of $\KK_n$ prime to $p$ and $(I_0 (\KK_n)/H_n)_p$ is the Sylow $p$-subgroup
of the finite group $I_0 (\KK_n)/H_n.$
Denote the group $Gal(\KK_n/\KK_0 )$ by $G_n$. It naturally acts
on $I(\KK_n)/H_n.$
Let us remind the reader a standard notation:
if a cyclic group $A$ generated by $a\in A$ acts on a group $B,$ then
$B^A=\{b\in B: a(b)=b\} .$

Our next goal is to prove
\begin{theorem}
$H^0 (G_n , I_0 (\KK_n)/H_n)=(I_0 (\KK_n)/H_n)^{G_n}=I_0 (\KK_0)/H_0.$
\end{theorem}
\begin{proof}
Obviously, it is sufficient to prove that
$$H^0 (G , I_0 (\KK_m)/H_m)=(I_0 (\KK_m)/H_m)^{G}=I_0 (\KK_{m-1})/H_{m-1},$$
where $G=Gal (\KK_m/\KK_{m-1}) .$

For this will need an idelic interpretation of $I_0 (\KK_m)$.
Let $J(k)$ be the idele group of a global field $k$.
For any valuation $v$, let $k_v$ be the corresponding completion.

Let $S$ be a finite subset of the set of all
valuations of $k.$

Denote by
$J^S (k)\subset J(k)$ the subgroup of elements which have $1$ at all $v$-components,
$v\in S$ and let $U^S (k)\subset J^S (k)$ be a subgroup, which has a unit of $k_v$ at
any $v$-component (clearly $v\notin S $).

In our case, for all $\KK_n$, $S$ will contain two elements: the archimedean valuation
and $\omega_n=(\zeta_{n+1} -1)$. We have a surjection $J^S (\KK_m)\to I_0 (\KK_m)$
with the kernel $U^S (\KK_m)$. Consequently, we have a surjection
$J^S (\KK_m)\to I_0 (\KK_m)/H_m$ with the kernel
$(\KK_m^* \cap U_{m,\omega_m })\times U^S (\KK_m)$. Here $\KK_m^*=\KK_m\backslash\{0\}, $
$U_{m,\omega_m }$ is the set of units of $(\KK_m)_{\omega_m}$, which are congruent $1$ {\it modulo}
$\omega_m^{p^{m+1}}=(\zeta_{m+1}-1)^{p^{m+1}}$,
and $\KK_m^*$ is embedded into $J^S(\KK_m)$ diagonally.

So, we have an exact sequence of $G$-modules:
$$0\to (\KK_m^* \cap U_{m,\omega_m })\to J^S (\KK_m)/U^S (\KK_m)\to I_0 (\KK_m)/H_m \to 0\ \ \ (1)$$

Let us prove first that
$H^0 (J^S (\KK_m)/U^S (\KK_m))=(J^S (\KK_m)/U^S (\KK_m))^G=J^S (\KK_{m-1})/U^S (\KK_{m-1}).$
To do this, we consider another exact sequence of $G$-modules:
$0\to U^S (\KK_m)\to J^S (\KK_m)\to J^S (\KK_m)/U^S (\KK_m)\to 0.$

Clearly, we have
\begin{itemize}
  \item $(U^S (\KK_m))^G=U^S (\KK_{m-1})$;
  \item $(J^S (\KK_m))^G=J^S (\KK_{m-1})$;
  \item $H^1 (G,U^S (\KK_m))=\{1\}$ because for all valuations $v\notin S$ of $\KK_m$,  local extensions $(\KK_m)_v /(\KK_{m-1})_{\tilde v}$ are unramified. Here the valuation
      $\tilde{v}$ of $\KK_{m-1}$ lies under $v$
  (see \cite{CF} for details).
\end{itemize}
Now, the corresponding exact sequence of cohomology groups reads as
$$0\to U^S (\KK_{m-1})\to J^S (\KK_{m-1})\to (J^S (\KK_m)/U^S (\KK_m))^G \to 0 ,$$
what proves that
$(J^S (\KK_m)/U^S (\KK_m))^G=J^S (\KK_{m-1})/U^S (\KK_{m-1}). $

Let us return to the exact sequence (1) and the corresponding exact cohomology sequence:
\begin{gather*}
0\to (\KK_m^* \cap U_{m,\omega_m })^G \to J^S (\KK_{m-1})\to (J^S (\KK_m)/U^S (\KK_m))^G \\
\to H^1 (G, \KK_m^* \cap U_{m,\omega_m })
\end{gather*}

The following observation is clear: $(\KK_m^* \cap U_{m,\omega_m })^G$ consists
of the elements of $x\in \KK_{m-1}^*$ such $x\equiv 1\ mod\ \omega_m^{p^{m+1}}$.
Since $\omega_m^{p^{m+1}}=\omega_{m-1}^{p^{m}}$, we deduce that
$(\KK_m^* \cap U_{m,\omega_m })^G=\KK_{m-1}^* \cap U_{m-1,\omega_{m-1} } .$

Thus, the statement of the theorem will follow from that of
$H^1 (G, \KK_m^* \cap U_{m,\omega_m })=\{ 1\} .$ Let us prove this.

Any $1$-cocycle is generated by some $x\in \KK_m^* \cap U_{m,\omega_m } $ such
that \newline
$Norm_{\KK_m/\KK_{m-1}}(x)=1. $ Hilbert 90 immediately implies that
$x=\sigma (a)/a$ for some $a\in\KK_m^*$ ($\sigma$ is a generator of $G$).

On the other hand, the same $x$ generates a $1$-cocycle of $Z^1 (G, U_{m,\omega_m }).$
Let us prove that $H^1 (G, U_{m,\omega_m })$ is trivial.
Indeed, let $u\in U_{m,\omega_m }.$ Then the series $log (u)$ converges  that provides an
isomorphism  $U_{m,\omega_m }\cong\omega_m^{p^{m+1}}\cdot O((\KK_m)_\omega)$. Here, $O(k_v)$
denotes the ring of integers of the local field $k_v$. Since $O((\KK_m)_\omega)$ is a free
$G$-module, we obtain $H^1 (G, U_{m,\omega_m })$ is trivial. Consequently, $x=\sigma (u)/u$
for some $u\in U_{m,\omega_m }.$

Now, it is not difficult to finish the proof of the theorem. We have
$x=\sigma (a)/a=\sigma (u)/u.$ Therefore, $t=a/u \in (\KK_{m-1})_\omega$.
We can find $a_1\in\KK_{m-1}$ and $u_1\in U_{m-1,\omega_{m-1} }$ such that $t=a_1 /u_1$
 because $\KK_{m-1}$ is dense in $( \KK_{m-1})_\omega$. Finally we get:
 $b=a/a_1=u/u_1 \in \KK_m^* \cap U_{m,\omega_m }$ and $\sigma (b)/b=\sigma (a)/a=x .$
 Thus, $H^1 (G, U_{m,\omega_m })$ is trivial and the theorem is proved.
\end{proof}
The main result of the section is
\begin{theorem}
$R_- - r_- \geq r_0 .$
\end{theorem}
\begin{proof}
The idea of the proof is to find $r_0$ cyclic extensions $\MM_\alpha$ of $\KK_n$ such that
\begin{enumerate}
  \item $\MM_\alpha\subset \MM\cdot\KK_n$;
  \item $Gal (\MM_\alpha/\KK_n)$ is an odd group with respect to the natural action of the
  complex conjugation;
  \item The extension $\MM_\alpha/\KK_n$ ramifies at $\omega_n .$
\end{enumerate}
There are obvious candidates for $\MM_\alpha$, namely
$\MM_\alpha=\KK_n (\epsilon^{1/p^{m_i +1}})$, where $\epsilon$'s are
units of $\KK=\KK_0$, which
are local $p^{m_i}$-powers ($m_i>0$) in $\KK_{\omega_0}$ but not $p$-powers in
$\KK .$ We have exactly $r_0$ such units. Let $m=m_i .$
Let $\gamma=\epsilon^{1/p^m}\in \KK_{\omega_0} .$
Then, the extension $\KK_{\omega_0}(\gamma^{1/p})/\KK_{\omega_0}$ ramifies and consequently
the extension $(\KK_n)_{\omega_n}(\gamma^{1/p})/(\KK_n)_{\omega_n}$ ramifies. Let us explain this.
\begin{explanation}
Indeed, this extension cannot become unramified of degree $p$ because it contains
a ramified subextension $\KK_{\omega_0}(\gamma^{1/p})/\KK_{\omega_0}$

Then, the only possibility is
that $(\KK_n)_{\omega_n}(\gamma^{1/p})=(\KK_n)_{\omega_n}$. Taking $n$ minimal with this property,
we get following relations between degrees:
$$
[(\KK_{n-1})_{\omega_{n-1}} (\gamma^{1/p}):(\KK_{n-1})_{\omega_{n-1}}]=
[(\KK_{n})_{\omega_n} :(\KK_{n-1})_{\omega_{n-1}}=p].
$$

Further,
$(\KK_{n-1})_{\omega_{n-1}} (\gamma^{1/p})\subset (\KK_{n})_{\omega_{n}} (\gamma^{1/p})$
and hence,
$(\KK_{n-1})_{\omega_{n-1}} (\gamma^{1/p})=(\KK_{n})_{\omega_{n}}=
(\KK_{n-1})_{\omega_{n-1}}(\zeta_{n}^{1/p}) .$
So,
$\zeta_n \gamma^{-1}=r^p ,\ r\in (\KK_{n-1})_{\omega_{n-1}} . $
The latter equality cannot be true even modulo
$(\zeta_n -1)^2$ that completes the explanation of the third statement.
\end{explanation}
The second statement is clear because by Kummer's Lemma $\epsilon$ is real and consequently,
$Gal (\KK_n (\epsilon^{1/p^{m+1}})/\KK_n)=(Gal (\KK_n (\epsilon^{1/p^{m+1}})/\KK_n))_- .$

It remains to prove the first statement.

{\bf Claim 1:} $\KK_n (\epsilon^{1/p^{m+1}})\subset\MM_n$.

To prove this claim we have to use some facts of the global class field theory (see \cite{CF}).
Let $F_{\KK_n  (\epsilon^{1/p^{m+1}})/\KK_n}: I_0 (\KK_n )\to Gal(\KK_n  (\epsilon^{1/p^{m+1}})/\KK_n)$
be the map constructed using Frobenius maps.
The Artin map (see again \cite{CF})
$\psi_{\KK_n  (\epsilon^{1/p^{m+1}})/\KK_n}: J(\KK_n)\to Gal (\KK_n  (\epsilon^{1/p^{m+1}})/\KK_n)$
satisfies
\begin{itemize}
  \item $\psi_{\KK_n  (\epsilon^{1/p^{m+1}})/\KK_n}(x)=\Pi_v \psi_v (x_v)$,
  here $x=(x_v)\in J(\KK_n )$ and $\psi_v$ are local Artin maps;
  \item $\psi_{\KK_n  (\epsilon^{1/p^{m+1}})/\KK_n}(x)=F_{\KK_n  (\epsilon^{1/p^{m+1}})/\KK_n}((x)^S)$
  where $x\in J(\KK_n)^S $ and $(x)^S $ is the image of $x$ under the canonical map
  $J(\KK_n)^S \to I_0 (\KK_n)$;
  \item $\psi_{\KK_n  (\epsilon^{1/p^{m+1}})/\KK_n}(\KK_n^*)=1$.
  \end{itemize}
 We proceed to the proof of Claim 1. Let $(q)\in H_n$. We have to prove that
 $F_{\KK_n  (\epsilon^{1/p^{m+1}})/\KK_n}((q))(\epsilon^{1/p^{m+1}})=\epsilon^{1/p^{m+1}} .$

Let us present $q\in\KK_n^*\subset J((\KK_n)$ as $q=(q)^S\cdot q_{\omega_n}$, where
$(q)^S=(q,\ldots q,\ldots )\in J(\KK_n)^S$ and $q_{\omega_n}=q\in(\KK_{n})_{\omega_n}$.
We have (denoting the Artin map by simply $\psi$, similarly for $F$) :
$$
1=\psi ((q)^S)\cdot \psi_{\omega_n}(q)=F((q))\cdot \psi_{\omega_n}(q), \
(q)\in H_n\subset I_0 (\KK_n).
$$
Therefore, proving $F((q))(\epsilon^{1/p^{m+1}})=\epsilon^{1/p^{m+1}} $ is
equivalent to
$\psi_{\omega_n}(q)(\epsilon^{1/p^{m+1}})=\epsilon^{1/p^{m+1}} .$ The latter is
equivalent to that of $(\epsilon,q)_{m+1} =1$. Here,
$(\cdot , \cdot )_{k}$ is the $\omega_n$-local
norm residue symbol with values in $p^{k}$-roots of unity. Taking into account
that $\epsilon=\gamma^{p^m}$ we deduce that $(\epsilon,q)_{m+1} =(\gamma,q)_1 .$
The latter is $1$ because of the following exercise from \cite{CF}:
$$
(1-\omega_n^k,1-\omega_n^l)_1=(1-\omega_n^k,1-\omega_n^{k+l})_1\cdot
(1-\omega_n^l,1-\omega_n^{k+l})_1^{-1}\cdot (1-\omega_{n}^{k+l}, \omega_n)_1^{-l}
$$
In our case it is sufficient to consider $k>1$ and $l\geq p^{n+1}$. Then
$1-\omega_n^{k+l}$ is a $p$-power and all three factors in the right hand side are
equal to $1.$ Claim 1 is proved.

{\bf Claim 2:} $\KK_n  (\epsilon^{1/p^{m+1}})\subset \MM\cdot\KK_n .$

Let $\sigma_n$ be a generator of $G_n=Gal(\KK_n /\KK)$.
We have proved that \newline $\KK_n  (\epsilon^{1/p^{m+1}})\subset \MM_n$ and
$H^0 (G_n , I_0 (\KK_n)/H_n)=(I_0 (\KK_n)/H_n)^{G_n}=I_0 (\KK_0)/H_0$

Following \cite{CF}, let us introduce $(I_0 (\KK_n)/H_n)_{G_n}$. For a moment, let us denote $G_n$ simply by $G.$
In general, if $A$ is any $G$-module, then $A_{G}$ is the largest factor-module of $A$ on which
$G$ acts trivially.

It was proved in \cite{CF}, Chapter 4, section 8 that the groups $A^{G}$ and $A_{G}$  (which in our case  are
 $I_0 (\KK_0)/H_0$ and some factor of $(I_0 (\KK_n)/H_n)$ respectively)
have the same order
(because $G$ is cyclic).
Moreover the norm map $Norm_{\KK_n /\KK_0 }$ induces a natural homomorphism $N: A_G\to A^G$.

Let us prove that in our case $N$ is an isomorphism. For this, it suffices to prove that $N$
is surjective.

Indeed, since the norm map $Norm_{\KK_n /\KK_0 }$ induces a surjection $Cl(\KK_n )\to Cl(\KK_0 )$,
it will be enough to prove that the map $Norm_{\KK_n /\KK_0 }$ induces a surjection
$N: P_0 (\KK_n )/H_n \to P_0 (\KK_0 )/H_0$. Here $P_0$ is the group of principal ideal co-prime to $p.$

The latter statement is an easy consequence of the fact proved in \cite{S1}:

$$  Norm_{\KK_n /\KK_0 } (a+b) \equiv Norm_{\KK_n /\KK_0 } (a) +Norm_{\KK_n /\KK_0 } (b) \ mod (p). $$

So,
$Gal (\KK_n \cdot \MM/\KK_n)$ is the largest factor of
$Gal(\MM_n/\KK_n)$ on which $G_n$ acts trivially.
Since
$\sigma_n (\epsilon)=\epsilon$, we deduce that
$\KK_n  (\epsilon^{1/p^{m+1}})\subset \KK_n \cdot\MM .$

The first statement is proved.

Hence, we have constructed $r_0$ extensions of
$\KK_n$
satisfying conditions 1,2,3 above. Corollary 2.2, Lemma 2.3, and Corollary 2.4
imply that the total number of such extensions is $R_- - r_- $ that implies
the required inequality.
The theorem is proved.
\end{proof}
\begin{corollary}
$R_- - r_- =r_0$
\end{corollary}
\begin{corollary}
$r_- =r_+ =R_+$
\end{corollary}
\begin{proof}
We have: $R_- - r_- =r_0=R-r =(R_- - r_- )+(R_+ - r_+) .$ Therefore, $R_+=r_+$.
The equality $R_+ =r_-$ was proved in the first section of the paper.
\end{proof}
\section{Proof of Vandiver's conjecture}
After the corollary above, in order to prove Vandiver's conjecture we have to prove
that $r_+ =0$. This will be done using Leopoldt's conjecture for the field $\KK$
(in this case it is known that the conjecture is true). Corollary 1.16 implies
\begin{lemma}
$(B_p )_- /((B_p )_- )^p \cong\FF_p^{\frac{p-3}{2}+r_-}$
\end{lemma}
In what follows, we will use the group $S_p$ defined in Theorem 1.10.
\begin{lemma}
Let a {\bf real} number $d\in\KK$ satisfy
$(d)=\alpha^p$ for some ideal $\alpha\subset\KK$.
Then $\KK (d^{1/p})\subset\MM .$
\end{lemma}
\begin{proof}
Exactly as in the proof of Claim 1 from the previous section, we have
to prove that $(d,q)_1 =1$ for any $q\equiv\ 1 mod(p)$. Since  $d$ is real, it satisfies
$d\equiv\ 1 mod (\zeta_1 -1)^2$ and the already mentioned in the previous section
exercise from \cite{CF} completes the proof.
\end{proof}
\begin{corollary}
Let $(d)=\alpha^p$ for some ideal $\alpha\subset\KK$ and $d$ be real.
Let the class of $\alpha$ is such that $   [\alpha ]\in (S_p)_+ .$
Then, such $d$ induces a character of $(B_p )_- /((B_p )_- )^p$ (Kummer's duality).
\end{corollary}
\begin{proof}
Clearly, $d$ induces a character of $(B_p ) /((B_p ))^p$. Since $d$ is real,
this character is a character of $(B_p )_- /((B_p )_- )^p$.
\end{proof}
\begin{theorem}
Let $(d)=\alpha^p$ for some ideal $\alpha\subset\KK$ such that $[\alpha ]$
is an element of $(S_p)_+ .$ Then $d$ induces a character of $(B_p )_- /((B_p )_- )^p$
and $d$ is a local $p$-power in $\KK_{\omega_0}$.
\end{theorem}
\begin{proof}
Let $\KK^+$ be the real subfield of $\KK$. The group $S_1 (\KK )$ was defined in the section 1.1 and
it was the $p$-Sylow component of the class group. Let us similarly define $S_1 (\KK^+ )$.
It was proved in \cite{KM} that the inclusion $\KK^+ \to \KK $ induces an
isomorphism $S_1 (\KK^+ )\to (S_1 (\KK ))_+$ (Lemma 4.5 and further references to papers by Iwasawa).
Moreover, it follows from the proof of Lemma 4.5 that the ideal $r\alpha, r\in\KK^* ,$ can be chosen to be the extension in $\KK$
of some ideal in $\KK^+$. It follows that $d=d_1r^p\varepsilon$, where $d_1$ is real and
$\varepsilon$ is a unit. Using the change $d\to\zeta_1 d$, we can achieve that
$d\equiv 1 \ mod (\zeta_1 -1)^2.$ Then $\varepsilon$ becomes a real unit and $d=d_2 r^p$
with real $d_2=d_1\varepsilon .$ Hence, Then $d$ induces a character of $(B_p )_- /((B_p )_- )^p$.

Let us prove that $d$ is locally a $p$-power. Indeed, the fact that $R_+ =r_+$ (Corollary 2.10)
implies that $d $ can be chosen $d\equiv 1 \ mod  (p)$. It was proved in \cite{S1}, Lemma 2, that
then $d\equiv 1 \ mod  (\zeta_1 -1)^p$. To complete the proof of the fact that $d$ is a local
$p$-power, we need one step further, namely to prove that $d\equiv 1\ mod  (\zeta_1 -1)^{p+1}$.
We have $d=d_2 r^p$, where $d_2$ is real. Changing $r\to\zeta_1 r$ we can achieve
$r\equiv 1\ mod (\zeta_1 -1)^2$ and  $r^p\equiv 1\ mod (\zeta_1 -1)^{p+1}$. Since
$d\equiv\ 1\ mod (\zeta_1 -1)^p$ and $d_2$ is real, we see that
$d_2\equiv 1\ mod (\zeta_1 -1)^{p+1}$ and consequently $d\equiv 1 \ mod (\zeta_1 -1)^{p+1} .$
\end{proof}
\begin{corollary}
The character group of $(B_p )_- /((B_p )_- )^p$ is generated by real units of $\KK$
(we have $\frac{p-3}{2}$  ones) and the set $\{ d_i\}$  defined above (we have $r_+ =r_-$ ones).
\end{corollary}
Now, we can prove Vandiver's conjecture.
\begin{theorem}
$r_+ =0.$
\end{theorem}
\begin{proof}
Let $U_1$ be a subgroup of $U(\Z_p [ \zeta_1 ])^+$ ("real" $p$-adic units) generated
by elements congruent $1\ mod (\zeta_1 -1)^2$. Let $W\subset U_1$ be the subgroup of elements
of norm $1$, i.e. $Norm_{\KK_{\omega_0}/\QQ_p}(w)=1 .$ It is well-known that $W$ is a
$\Z_p$-module of rank $\frac{p-3}{2}$. Let $E$ be a group of real units of $\Z [\zeta_1 ]$
congruent $1\ mod (\zeta_1 -1)^2$
and let $\bar{E}$ be its closure in $W.$ By Leopoldt's conjecture $\bar{E}$ is a
$\Z_p$-module of rank $\frac{p-3}{2}$. Therefore, $W/\bar{E}$ is a finite $\Z_p$-module
(of rank at most $\frac{p-3}{2}$).

Let us "improve" $d_i$ from the corollaries above. We have:
$Norm_{\KK/\QQ}(d_i)=q_i^p $  (Explanation:  $(d_i)=\alpha_i^p$;
$Norm_{\KK/\QQ}(\alpha_i)=(t_i)=\pm t_i\in\QQ$ and so, $q_i=t_i$ or $q_i=-t_i$).
Then $Norm_{\KK/\QQ}(d_i^{p-1}q_i^{-p})=q_i^{p(p-1)}q_i^{-p(p-1)}=1 . $
Consequently, $h_i=d_i^{p-1}q_i^{-p}\in W$ and $h_i$ generate a character of
$(B_p )_- /((B_p )_- )^p$.

Since $h_i$ is a local $p$-power, we can find $r_i\in \KK$ with $Norm_{\KK/\QQ}(r_i)=1$
such that $h_i r_i^p\in \bar{E}$, which  is generated
by real units of $\Z[\zeta_1 ]$ as a $\Z_p$-module.

By Kummer's duality, the group of characters of
$(B_p )_- /((B_p )_- )^p\cong \FF_p^{\frac{p-3}{2} +r_-}$ is a subgroup of $\KK^*/(\KK^*)^p$.

On the other hand, the "dual" elements are $\{ \epsilon_k \}_{k=1}^{\frac{p-3}{2}}$ and
$ \{ q_i= r_i^{pt_i}d_i \}_{i=1}^{r_+}$ and they are elements of $\bar{E}$.

Therefore, the group of characters of
$(B_p )_- /((B_p )_- )^p$ is a subgroup of
$$\frac{\bar{E}\cap \KK^*}{\bar{E}\cap (\KK^*)^p} .$$

Our next aim is to show that the images of
$\{ q_i= r_i^{pt_i}d_i \}_{i=1}^{r_+}$ in
$\frac{\bar{E}\cap \KK^*}{\bar{E}\cap (\KK^*)^p} $ are contained in the subgroup
generated by the images of $\{ \epsilon_k \}_{k=1}^{\frac{p-3}{2}}$.

Let $G=Gal (\KK /\Q) =\Z/(p-1)\Z .$ Clearly, $\KK^*$, $\bar{E}$ and $\frac{\bar{E}\cap \KK^*}{\bar{E}\cap (\KK^*)^p} $
are $G$-modules. Moreover, the latter two groups are $\Z_p$-modules.

Let $\varepsilon_k ,\ k=2,4,\ldots,p-3$ be idempotents of $\Z_p [G] .$

Without loss of generality we may assume that $q_i\in \varepsilon_k\bar{E}$.
By Leopoldt's conjecture, which is true in our case,
$\varepsilon_k\bar{E}\cong \Z_p$ is generated by $\epsilon_k$ as a $\Z_p$-module.
In other words, $q_i =\epsilon_k^{\sum_{t=0}^{\infty} a_t p^t}$, where $a_t =0,1,\ldots, p-1 .$

Let us introduce elements $q_{i,N}=\epsilon_k^{\sum_{t=1}^{N} a_t p^t}$.
Then $q_i =\epsilon_k^{a_0} q_{i,N} Q_N$, where $Q_N\in p^{N+1}\Z_p .$

Further, all the elements $Q_N$ are contained in $\KK^* \cap \bar{E}$  and generate one and the same character of
$(B_p )_- /((B_p )_- )^p $ namely the one generated by $q=q_i \epsilon_k^{-a_0}$. Since $Q_N\in p^{N+1}\Z_p $
it follows that the image of $q$ in $\frac{\bar{E}\cap \KK^*}{\bar{E}\cap (\KK^*)^p} $ is contained
in  the image of $(\varepsilon_k \bar{E})^{p^N}\cap\KK^* $ in the same factor
$\frac{\bar{E}\cap \KK^*}{\bar{E}\cap (\KK^*)^p} $
for all $N$. Then clearly
$q=1$ and $q_i =\epsilon_k^{a_0}$ in the character group of $(B_p )_- /((B_p )_- )^p$.

Therefore, $(B_p )_- /((B_p )_- )^p \cong \FF_p^{\frac{p-3}{2}}$ and $r_+ =R_+ = r_- =0$.
 The theorem and Vandiver's conjecture are proved.


 \end{proof}
\section{Two appendices}
I am thankful to the readers of the paper "Vandiver's Conjecture via K-theory" who asked me a number of important
questions. As a result of that questions, in the main text of the paper I replaced the  proof of of Theorem 3.6 with a new one.
Additionally, I decided to publish some new proofs of certain statements in the main text.

\subsection{Appendix 1, new proof of Theorem 3.4.}

\begin{theorem}
Let $(d)=\alpha^p$ for some ideal $\alpha\subset\KK$ such that $[\alpha ]$
is an element of $(S_p)_+ .$ Then $d$ induces a character of $(B_p )_- /((B_p )_- )^p$
and $d$ is a local $p$-power in $\KK_{\omega_0}$.
\end{theorem}
\begin{proof}
We will present a proof not referring to \cite{KM}.

Let $\KK^+$ be the real subfield of $\KK$. The group $S_p (\KK )$ was defined in the section 1.1, Theorem 1.10.
Let $\alpha\in (S_p )_+$. This means that $\bar{\alpha}=r\alpha$, where $\bar{\alpha}=c(\alpha)$ is a result of the complex
conjugation applied to $\alpha\subset\KK$ and $r\in\KK $ (we abuse notations denoting an ideal and its image in the ideal class group
by the same letter, this should not lead to misunderstandings).
Consequently, we get $\bar{d}=\epsilon r^p d$ for some unit $\epsilon .$ Clearly, without loss of generality it is sufficient to
consider the case $d\equiv 1 mod (\zeta -1) .$ Furthermore, using transformation $d\to \zeta^k d$, it suffices to consider
$d\equiv 1 mod (\zeta -1)^2 .$ Then it follows that $\epsilon$ will become a real unit.
Therefore, $Norm_{\KK/\KK^+ }(\epsilon r^p)=\epsilon^2 (\bar{r}r)^p=1$ and $\bar{r}r$ is also a unit. This implies
that $\epsilon=\gamma^p$, $\bar{d}= r_1^p d$ and $Norm_{\KK/\KK^+ }(r_1) =1$. By Hilbert 90, $r_1=r_2 / c({r_2})$.
Hence, $c{(dr_2^p )} =dr_2^p$ and $ dr_2$ is a real number.
Then $dr_2^p$ induces a character of $(B_p )_- /((B_p )_- )^p$ by Corollary 3.3.

Let us prove that $d_1=dr_2^p$ is locally a $p$-power. Indeed, the fact that $R_+ =r_+$ (Corollary 2.10)
implies that $d_1 $ can be chosen $d_1\equiv 1 \ mod  (p)$. It was proved in \cite{S1}, Lemma 2, that
then $d_1 \equiv 1 \ mod  (\zeta_1 -1)^p$. To complete the proof of the fact that $d_1$ is a local
$p$-power, we need one step further, namely to prove that $d\equiv 1\ mod  (\zeta_1 -1)^{p+1}$.
The latter is clear because $d_1$ is real.

\end{proof}
\subsection{Appendix 2. Complement to Explanation 2.8}
Let $\gamma=\epsilon^{1/p^m}\in \KK_{\omega_0} ,$ where $\epsilon$ is a unit of $\KK$ and a local (with respect to $\omega_0$)
$p^m$-power.
\begin{lemma}
Let $x=\bar{x}\in\KK_{\omega_0}$ (we say that $x$ is real)  be a $p$-power, i.e. there exists $a\in\KK_{\omega_0}$ such that $a^p=x$.
Then we can choose $a$ such that $a=\bar{a}$.
\end{lemma}
\begin{proof}
If the statement of Lemma is not true, then $\zeta a=\bar{a}$ because $x=\bar{x}$.
It is easy to check that  $a_1=a\zeta^{\frac{p+1}{2}}$ is what we need.

\end{proof}

\begin{corollary}
$\gamma $ can be chosen real.

\end{corollary}
In this Appendix, let us denote ${(\KK_{n} )}_{\omega_n}$ by $k_n$.
\begin{lemma}
$k_2\cap\KK_{\omega}(\gamma^{1/p})=\KK_{\omega}$ if $\gamma$ is real.
\end{lemma}
\begin{proof}
If it is not true, then $k_2=\KK_{\omega}(\gamma^{1/p})$. This implies $\gamma=\zeta r^p$ for some $r\in\KK_{\omega}.$
However, this equality does not hold even $mod (\zeta -1)^2 .$
\end{proof}
Now, assume that $\KK_{\omega}(\gamma^{1/p})/\KK_{\omega}$ ramifies and $\gamma$ is real.
\begin{theorem}
$k_2 (\gamma^{1/p})/k_2$ ramifies.
\end{theorem}
\begin{proof}
Assume that our extension is unramified. Since the only unramified extension of $k_2$ of degree $p$ is
$k_2 ((1+\omega^p)^{1/p})$, we conclude that $(1+\omega^p)r_2^p=\gamma$ for some $r_2\in k_2$.
Further, $r_2$ cannot be an element of $\KK_{\omega}$ because then the extension
$\KK_{\omega}(\gamma^{1/p})/\KK_{\omega}$ is unramified while it ramifies.

Hence, we have $r_2^p \in\KK_{\omega}$ while $r_2 \notin \KK_{\omega}$. Then we have $k_2 =\KK_{\omega} (r_2)$
and
$$\zeta s^p =\frac{\gamma}{1+\omega^p}, \ s\in\KK_{\omega} .$$
Again, since $\gamma$ is real, this equality does not hold $mod(\zeta -1)^2$
\end{proof}
\begin{corollary}
The extension $k_n (\gamma)/k_n$ ramifies.
\end{corollary}
\begin{proof}
If we substitute $\KK_{\omega}$ by $k_2$, $\zeta \to \zeta_2$, and $k_2$ by $k_3$  in the proof of the theorem, then it is easy
to see that all the arguments of the proof will work. Thus, the extension $k_3 (\gamma)/k_3$ ramifies.
An obvious induction completes the proof of Corollary.
\end{proof}

\end{document}